\documentclass[12pt,centertags,oneside]{amsart}
\usepackage[textwidth=13cm,textheight=21cm]{geometry}

\usepackage{amsmath,amstext,amsthm,amssymb,amsfonts,amscd}
\usepackage{mathrsfs}
\usepackage[colorlinks=true]{hyperref}
\usepackage{graphicx}
\usepackage[T1]{fontenc}
\usepackage[latin1]{inputenc}
\usepackage{color,tocvsec2}
\usepackage{typearea}
\usepackage{charter}
\usepackage{enumerate}
\usepackage{lscape}
\usepackage[all]{xy}
\usepackage[normalem]{ulem}

\usepackage{multicol}        
\makeindex             

\theoremstyle{plain}
\newtheorem{lem}{Lemme}[section]
\newtheorem{thm}[lem]{Theorem}
\newtheorem{prop}[lem]{Proposition}
\newtheorem{cor}[lem]{Corollary}

\theoremstyle{definition}
\newtheorem{defin}[lem]{Definition}

\theoremstyle{remark}
\newtheorem{re}[lem]{Remark}

\numberwithin{equation}{section}
\numberwithin{figure}{section}

\newcommand{\cE}{\mathcal{E}}

\newcommand{\cO}{\mathcal{O}}

\newcommand{\bZ}{\mathbf{Z}}
\newcommand{\bR}{\mathbf{R}}
\newcommand{\bC}{\mathbf{C}}

\DeclareMathOperator{\ind}{\mathrm ind}

\DeclareMathOperator{\Tr}{\mathrm Tr}

\renewcommand{\Re}{\mathrm{Re}\,}

\DeclareMathOperator{\End}{\mathrm End}

\DeclareMathOperator{\rk}{\mathrm rk}
\DeclareMathOperator{\GL}{\mathrm GL}

\newcommand{\<}{\langle}
\renewcommand{\>}{\rangle}
\newcommand{\ol}{\overline}

\newcommand{\p}{\partial}

\renewcommand{\(}{\left(}
\renewcommand{\)}{\right)}
\renewcommand{\[}{\left[}
\renewcommand{\]}{\right]}
\renewcommand{\l}{\leqslant}
\newcommand{\g}{\geqslant}
\newcommand{\e}{\epsilon}

\newcommand{\bbS}{\mathbb{S}}

\begin{document}

\title{Morse-Smale flow, Milnor metric, and  dynamical zeta function}
\author{Shu Shen}
\author{Jianqing Yu}
\address{Institut de Math\'ematiques de Jussieu-Paris Rive Gauche,
Sorbonne Universit\'e,  4 place Jussieu, 75252 Paris Cedex 05, France.}
\email{shu.shen@imj-prg.fr
}
\address{School of Mathematical Sciences, University of Science and Technology of China, 96 Jinzhai Road, Hefei, Anhui 230026, P. R. China.}
\email{jianqing@ustc.edu.cn}
\thanks {}


\dedicatory{}

\begin{abstract}	
	We introduce  a 
	Milnor metric on the determinant  line of  the cohomology of the
	underlying  closed manifold with coefficients in a flat vector 
	bundle, by means of  interactions  between the fixed points and the closed orbits of a Morse-Smale	flow,.
	This allows us to generalise the notion of the absolute value at zero 
	point of the 
	Ruelle dynamical zeta function, even in the case where  this 
	value is not well-defined in the classical  sense. 	We give a
	formula relating the Milnor metric and the Ray-Singer metric. An 
	essential ingredient of our proof is  Bismut-Zhang's Theorem.
\end{abstract}

\maketitle
\tableofcontents
\settocdepth{section}
\section*{Introduction}
The study of the relation between the combinatorial/analytic torsion  of 
a flat vector bundle and the Morse-Smale flow is initiated by Fried \cite{Friedconj} and
S\'anchez-Morgado \cite{MorgadoMorseSmale}.
In this paper, we give a formula relating
\begin{itemize}
	\item  a spectral invariant: the Ray-Singer metric associated
	with  a  flat  vector bundle with a Hermitian metric on a closed Riemannian
	manifold;
	\item  a dynamical  invariant: the Milnor metric which reflects 
	the interactions 	between the fixed points and the  
	closed orbits of the Morse-Smale flow, and  generalises the  
	absolute value at zero point of the Ruelle dynamical zeta function;
	\item a transgressed Euler class:
the Mathai-Quillen current. 
\end{itemize}

 \subsection{Background}
 Let $X$ be a connected closed   smooth   manifold of dimension $m$. Let 
 $(F,\nabla^{F})$ be a complex
 flat vector bundle of rank $r$ on $X$ with flat connection 
 $\nabla^{F}$. Let $\rho: \pi_1(X)\rightarrow \GL_r(\bC)$ be the 
 holonomy representation of the
fundamental group $\pi_1(X)$.
%
%
Denote by $H^{\scriptscriptstyle\bullet}(X,F)$ the cohomology of
the sheaf of locally constant sections of $F$, and by 
$\lambda=\bigotimes_{i=0}^{m}\(\det H^{i}(X,F)\)^{(-1)^{i}}$ the
determinant line of $H^{{\scriptscriptstyle\bullet}}(X,F)$.

{Assume that  $H^{{\scriptscriptstyle\bullet}}(X,F)=0$ and that $F$ is equipped 
with a flat metric, which is equivalent to say that  its holonomy
representation $\rho$ is unitary.  
The Reidemeister torsion  \cite{ReidemeisterTorsion,FranzTorsion,dRTorsion}  is a positive real number 
defined with the  help of a triangulation on $X$. However, it does not depend on the 
triangulation and becomes a topological invariant.  It is the first 
invariant that could distinguish closed manifolds such as lens spaces 
which are homotopy equivalent but not homeomorphic.}

The analytic torsion was introduced by Ray and Singer  
\cite{RSTorsion} as an analytic counterpart of the Reidemeister torsion. 
In order to define the analytic torsion one has to choose a 
Riemannian metric on $X$. The analytic torsion is a certain 
weighted alternating product of regularized determinants of the Hodge 
Laplacians  acting on the space of differential forms  with values in 
$F$. 

The celebrated Cheeger-M\"uller Theorem
\cite{Ch79,Muller78} tells us that the Ray-Singer analytic torsion 
coincides with the Reidemeister combinatorial torsion. Bismut-Zhang \cite{BZ92} and M\"uller \cite{Muller2} simultaneously
considered generalisations of this result.
M\"uller \cite{Muller2} extended his result to the case where $F$ is 
unimodular, i.e., $|\det 
\rho(\gamma)|=1$ for all $\gamma\in \pi_{1}(X)$. Bismut and Zhang \cite[Theorem 0.2]{BZ92}  generalised the original Cheeger-M\"uller
Theorem to arbitrary flat vector bundles with arbitrary Hermitian
metrics.  
There are also various extensions to the 
equivariant case  by Lott-Rothenberg \cite{LottRothengerg}, L\"uck \cite{LuckTorsion}, and Bismut-Zhang  \cite{BZ94}, to the family case by 
Bismut-Goette  \cite{BG01}, and to 
manifolds with 
boundaries by Br\"uning-Ma \cite{BruningMa_2012}.

Let us explain Bismut-Zhang's theorem \cite[Theorem 0.2]{BZ92} in more detail. Indeed, to 
formulate their result in the case  where the flat vector bundle is not 
necessarily  acyclic or unitarily flat, Bismut and Zhang introduced the so-called Ray-Singer metric, which  is a 
metric on $\lambda$  defined as the product of  
the analytic torsion with an $L^2$-metric on $\lambda$. 
Also they introduced the Milnor metric on $\lambda$ which is  a combinatorial  
metric associated with a  Morse-Smale gradient  flow. It generalises the 
Reidemeister torsion to  the case where $F$ is neither  acyclic 
nor unitarily flat. In this way, 
they were able to extend the Cheeger-M\"uller Theorem to a comparison theorem of two metrics on $\lambda$, one is analytic and the other 
one is combinatorial.

%
%
%
%

The study of the relation between the combinatorial/analytic torsion  and the dynamical system
can be traced  back to Milnor \cite{MilnorZcover}.  Fried
\cite{FriedRealtorsion}  showed that on hyperbolic manifolds the analytic torsion of an acyclic unitarily flat vector
bundle is equal to the  value at zero point   of the Ruelle dynamical zeta
function of the geodesic flow.  He conjectured 
\cite[p.\,66~Conjecture]{Friedconj} that similar results should hold true for  more general
flows. In \cite{Shfried}, following previous contributions 
by
Moscovici-Stanton \cite{MStorsion}, using Bismut's orbital integral formula 
\cite{B09}, the  author  affirmed  the Fried
conjecture for  geodesic flows on  closed locally symmetric 
manifolds. In \cite{Shen_Yu}, the authors made a further 
generalisation to closed locally symmetric orbifolds.

Besides the gradient flow, Morse-Smale flow is the simplest structurally stable dynamical system
which has only two  types of  recurrent behaviors: closed orbits and
fixed points \cite{Palis68,PalisSmale70}.  Fried \cite[Theorem 3.1]{Friedconj}
proved  his conjecture for the Morse-Smale flows without fixed points. 
When compared with Bismut-Zhang's Theorem  \cite[Theorem 0.2]{BZ92}, 
it seems  natural  to ask 
whether there is a relation between the torsion invariant (or more generally the
Ray-Singer metric for non acyclic and non unitarily flat vector bundle) and a general Morse-Smale flow which has both
fixed points and closed orbits.

This is one of the motivations  of S\'anchez-Morgado's work
\cite{MorgadoMorseSmale}. He showed that the
heteroclinic orbits have a non trivial contribution in the torsion 
invariant, and in this way he constructed  a counterexample to Fried's 
conjecture on Seifert manifolds. 

In this
paper, we introduce a new Milnor metric, 
which indeed contains the heteroclinic  contributions and 
generalises  the absolute  value at zero point of the Ruelle dynamical zeta function, and we 
give a comparison theorem for the Milnor and Ray-Singer metrics on 
$\lambda$. We believe 
that in this way we  give  a complete answer in the affirmative to the above  question.

Let us  mention that there is another interpretation of the Ruelle dynamical zeta 
function  provided by Dang-Rivi\`ere \cite{DR_MSR}. See also 
\cite{DR_MSG, DR_MSI, DR_MSII, DR_MSW} for related works. 

 \subsection{A new Milnor metric}
  A vector field $V$ is called Morse-Smale if $V$ generates a flow
  whose  nonwandering set is the union of a finite set $A$ of
  hyperbolic fixed points and a finite set $B$ of hyperbolic closed  orbits, and if the stable and unstable manifolds of the critical
  elements in $A\coprod B$ intersect transversally.

  Let us take a Hermitian metric $g^{F}$ on $F$. In Section 
  \ref{sSF},  we construct on $\lambda$   a   Milnor type metric
$\|\cdot\|^{M,2}_{\lambda, V}$ using  long exact 
sequences associated with a Smale 
filtration of the Morse-Smale flow.  Note that 
the long exact 
sequences  encode the information about the interactions
between the critical elements in $A\coprod B$.  %
%
 If $V$ is a negative gradient of a Morse function, 
then our Milnor metric is just the classical one as defined in  
 \cite[Definition 1.9]{BZ92}, which generalises  \cite{MilnorWhiteheadTor}.
 

Our first result says that  the  Milnor metric 
$\|\cdot\|^{M,2}_{\lambda, V}$ is a generalisation of  the absolute value at 
zero point of 
the Ruelle dynamical zeta function. For a closed orbit $\gamma\in B$, 
let $\ell_{\gamma}\in \bR^{*}_{+}$ be its minimal period, and let 
$\ind(\gamma)\in \mathbf{N}$  be its index 
(see \eqref{indexg}). Take  $\Delta(\gamma)$ to be $1$ if 
$\gamma$ is untwist and $-1$ in the contrary  case (see \eqref{Delta}). The  Ruelle dynamical zeta function 
 is defined for $s\in \bC$ by 
\begin{align}\label{intoR}
	R_{\rho}(s)=\prod_{\gamma\in
	B}\det\(1-\Delta(\gamma)\rho(\gamma)e^{-s
	\ell_\gamma}\)^{(-1)^{\ind(\gamma)}}.
\end{align}

\begin{prop}\label{prop}
	If $V$ does not have any fixed points, and if none of $\Delta(\gamma)$ is an eigenvalue of $\rho(\gamma)$, then 
$H^{{\scriptscriptstyle\bullet}}(X,F)=0$, and the norm of the canonical section 
$\mathbf{1}\in \bC=\lambda$ is given by 
\begin{align}
	\|\mathbf{1}\|^{M}_{\lambda,V}=|R_{\rho}(0)|^{-1}. 
\end{align}
\end{prop}

\subsection{The main result of the paper}
Let $g^{TX}$ be a metric on $TX$. Let
$\psi(TX,\nabla^{TX})$ be the Mathai-Quillen current  associated with the
Levi-Civita connection $\nabla^{TX}$ (see Section \ref{sMQ}). It is a
current of degree $m-1$ defined
on the total space of the tangent bundle $TX$, which takes values in 
$o(TX)$, the orientation line bundle of $TX$. Let $\|\cdot\|^{RS,2}_{\lambda}$ be the Ray-Singer metric
on $\lambda$ associated with $(g^{TX},g^{F})$ (see Section \ref{sRS}).
Set $	
\theta(F,g^{F})=\Tr\[(g^{F})^{-1}\nabla^{F}g^{F}\]\in \Omega^{1}(X)$.  
Our main result  is the following.  

\begin{thm}\label{thm1} We have
	\begin{align}\label{eqthm1}
	\log
	\(\frac{\|\cdot\|^{RS,2}_{\lambda}}{\|\cdot\|^{M,2}_{\lambda,V}}\)=-\int_{X}\theta\(F,g^{F}\)(-V)^{*}\psi\(TX,\nabla^{TX}\).\end{align}
\end{thm}

If $V$ does not have any closed orbits, Theorem \ref{thm1} reduces to 
\cite[Theorem 0.2]{BZ92}. Note also that if $F$ is unitarily flat, then the right-hand side of 
\eqref{eqthm1} varnishes. Therefore, if $V$ does not have any fixed  
points and if $F$ is unitarily flat, by  Proposition \ref{prop},  our theorem corresponds to \cite[Theorem 3.1]{Friedconj}.

Our proof of Theorem \ref{thm1} is based on a result of Franks
\cite[Proposition 5.1]{Franks1979}, who
constructed  a gradient flow by destroying the
closed orbits of the Morse-Smale flow.
In Section \ref{sCMM}, we first
establish a comparison formula between 
our Milnor metric 
associated with the original Morse-Smale flow
and the classical one associated with Franks'
gradient flow. In Section \ref{S2}, to obtain Theorem \ref{thm1}, we apply Bismut-Zhang's formula \cite[Theorem
0.2]{BZ92}, which compares the 
Ray-Singer metric with the
Milnor metric for Franks' gradient flow.

Recall that $F$ is said to be  unimodular, if its holonomy representation
$\rho$ is unimodular, i.e., $|\det \rho(\gamma)|= 1$ for all 
$\gamma\in \pi_{1}(X)$.  This is
equivalent to the fact that there is a Hermitian metric $g^{F}$
such that $\theta\(F,g^{F}\)=0.$
By Theorem \ref{thm1}, we get
\begin{cor}
	If $(F,g^{F})$ is unimodular, then
	\begin{align}
	\|\cdot\|^{RS,2}_{\lambda}=\|\cdot\|^{M,2}_{\lambda,V}.
\end{align}
\end{cor}

\subsection{Organisation  of the paper}
In Section \ref{Sp1}, we
introduce some conventions on the determinant line, the cohomology
of a circle, and also a long exact sequence associated with three manifolds $Y_{1}\subset Y_{2}\subset Y_{3}$.  In Section \ref{S1M}, we recall some background on Morse-Smale
flows. We also introduce the Milnor type metric, and we show Proposition 
\ref{prop}. In
Section \ref{S2}, we recall the constructions of Mathai-Quillen 
current and Ray-Singer metric. We show our main result.
We use the convention $\mathbf{N}=\{0,1,2,\ldots\}$ and $\mathbf{R}_{+}^{*}=(0,\infty)$. 

\subsection{Acknowledgement}
We are indebted to  Xiaolong Han and Xiaonan Ma for reading  a preliminary version of this paper and for useful suggestions. S.S. would like to thank  Nguyen Viet Dang and  Gabriel  Rivi\`ere for 
fruitful  discussions  on Morse-Smale flows. This paper was written 
while J.Y. was visiting the Institut de
Math\'ematiques de Jussieu-Paris Rive Gauche in Spring 2018.
The hospitality of the Institute and the partial financial support
from NSFC under Grant No.~11771411 are gratefully acknowledged.

\settocdepth{subsection}
\section{Preliminary} \label{Sp1}
This section is organised as follows. In Section \ref{sDL}, we
introduce our convention on the determinant line. In Section
\ref{sCS1}, we give a metric on the determinant line of  the 
cohomology of $\mathbb{S}^{1}$. This is  our model case near the 
closed orbits of a flow. In Section \ref{Fp}, we explain a long exact 
sequence associated with a triple  of manifolds $Y_{1}\subset 
Y_{2}\subset Y_{3}$. 

\subsection{The determinant line}\label{sDL}
Let $W$ be a complex finite dimensional vector space. We denote by
$W^*$ the dual space. If $\dim W=1$, we write $W^{-1}=W^{*}$.  Take
$\Lambda^{\scriptscriptstyle\bullet}(W)=\bigoplus_{j=0}^{\dim W}\Lambda^{j}(W)$ to be the exterior
algebra.

Set
\begin{align}
\det W =\Lambda^{\dim W} (W).
\end{align}
Clearly, $\det W$ is a complex line. If $W=\{0\}$, then  
\begin{align}
	\det 
W=\mathbf{\bC}. 
\end{align}
Let $E^{{\scriptscriptstyle\bullet}}=\oplus_{i\in \bZ} E^{i}$ be a finite dimensional  $\mathbf{Z}$-graded
space. Put
\begin{align}
	\det E^{{\scriptscriptstyle\bullet}}= \bigotimes_{i\in \bZ} \(\det E^{i}\)^{(-1)^{i}}.
\end{align}

For $m\in \mathbf{N}$, let 
\begin{align}
	(C^{{\scriptscriptstyle\bullet}},d): 0\to C^{0}\to C^{1}\to\cdots\to C^{m}\to 0
\end{align}
be a complex of finite dimensional vector 
spaces. By \cite{KM76} or 
\cite[(1.5)]{BGS1}, we 
have the canonical isomorphism of lines 
\begin{align}\label{detCdetH}
\tau_{C^{{\scriptscriptstyle\bullet}}}:	\det C^{{\scriptscriptstyle\bullet}}\simeq \det H^{{\scriptscriptstyle\bullet}}(C^{{\scriptscriptstyle\bullet}},d). 
\end{align}
If  $s_{j}^{k}\in C^{k}$  such that $\{s_{j}^{k}\}_{j=1}^{k_{j}}$ projects to a basis of 
$C^{k}/\ker (d_{|C^{k}})$, if $\mu_{j}^{k}\in \ker( d_{|C^{k}})$ such that  
$\{\mu_{j}^{k}\}_{j=1}^{k_{j}'}$ projects to a basis of 
$H^{k}(C^{\scriptscriptstyle\bullet},d)$,
then 
\begin{align}\label{Maja1}
	\(\wedge_{j} s^{0}_{j}\otimes \wedge_{j}\mu^{0}_{j}\)\otimes 
	\(\wedge_{j} (ds^{0}_{j})\otimes 
	\wedge_{j}s^{1}_{j} \otimes \wedge_{j}\mu^{1}_{j}\)^{-1} \otimes\cdots \otimes\(\wedge_{j} 
	(ds^{m-1}_{j})\otimes \wedge_{j}\mu^{m}_{j}\)^{(-1)^m}
\end{align}
defines a canonical element of $\det C^{\scriptscriptstyle\bullet}$. 
If $\ol{\mu}_{j}^{k}$ denotes the image of $\mu_{j}^{k}$ in $H^{k}(C^{{\scriptscriptstyle\bullet}},d)$, under 
\eqref{detCdetH}, the element \eqref{Maja1} maps to
\begin{align}\label{Maja11}
	\( \wedge_{j}\ol{\mu}^{0}_{j}\)\otimes 
	\( \wedge_{j}\ol{\mu}^{1}_{j}\)^{-1} \otimes\cdots \otimes\( 
	\wedge_{j}\ol{\mu}^{m}_{j}\)^{(-1)^m}\in \det 
H^{{\scriptscriptstyle\bullet}}(C^{{\scriptscriptstyle\bullet}},d). 
\end{align}

%
%

\subsection{The cohomology of $\bbS^{1}$}\label{sCS1}
Let $\mathbb{S}^{1}=\mathbf{R}/\mathbf{Z}$ be an oriented circle. Let $F$ be a flat vector 
bundle of rank $r$ on $\mathbb{S}^{1}$. Let $\rho: \pi_{1}(\bbS^{1})\to  
\GL_{r}(\bC)$ be the holonomy\footnote{For any flat vector bundle  
$F$ on a manifold $X$, the holonomy is a  representation $\rho: 
\pi_{1}(X)\to \mathrm{GL}_{r}(\bC)$ of the fundamental group 
$\pi_{1}(X)$ such that $F=\pi_{1}(X)\backslash (\widetilde{X}\times 
\bC^{r})$, where $\widetilde{X}$ is the universal  cover of $X$ and  $\pi_{1}(X)$ acts on the left on $\widetilde{X}$ by 
the deck transformation 
and on  $\bC^{r}$ by $\rho$.} of $F$. Let $a_{0}\in \pi_{1}(\bbS^{1})$ be the  
generator of $\pi_{1}(\bbS^{1})$, which is compatible with the 
orientation on $\mathbb{S}^{1}$. Set $A=\rho(a_{0})\in 
\mathrm{GL}_{r}(\bC)$.

%




Consider the canonical  triangulation on $\bbS^{1}$  induced by one $0$-simplex 
$\sigma_{0}$ and one $1$-simplex $\sigma_{1}$ as in  Figure \ref{fig:0}. 
\begin{figure}[htbp] 
\centering
\includegraphics[width=1in]{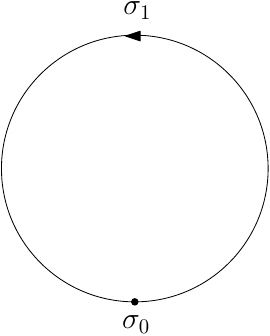} 
\caption{A triangulation on $\bbS^{1}$.}\label{fig:0} 
\end{figure} 
It induces a complex of simplicial cochains with values in  $F$ given 
by 
\begin{align}\label{CS1}
\(C^{{\scriptscriptstyle\bullet}}(\bbS^{1},F),d\):		\xymatrixcolsep{3pc}\xymatrix{
0  \ar[r] &\bC^{r}
\ar[r]^{A-1} &\bC^{r} \ar[r]          &0.}
\end{align}
By \eqref{detCdetH}, the canonical element $
\mathbf{1}\in \bC=\det C^{{\scriptscriptstyle\bullet}}(\bbS^{1},F)$ defines  an element
\begin{align}\label{eqsigmaa}
\sigma_{A}\in 	\det H^{{\scriptscriptstyle\bullet}}\(\mathbb{S}^{1},F\).
\end{align}
We equip $\det H^{{\scriptscriptstyle\bullet}}\(\mathbb{S}^{1},F\)$ with a metric  
$\|\cdot\|^{2}_{\det H^{{\scriptscriptstyle\bullet}}\(\mathbb{S}^{1},F\)}$ such that 
\begin{align}\label{nsa}
	\|\sigma_{A}\|_{\det H^{{\scriptscriptstyle\bullet}}\(\mathbb{S}^{1},F\)}=1
\end{align}

If $1$ is not an eigenvalue of $A$, then $H^{{\scriptscriptstyle\bullet}}(\bbS^{1},F)=0$. 
By \eqref{Maja1},  the norm of the canonical element 
$\mathbf{1}\in \bC=\det H^{{\scriptscriptstyle\bullet}}(\bbS^{1},F)$ is given by 
\begin{align}\label{sa}
	\|\mathbf{1}\|_{\det 
	H^{{\scriptscriptstyle\bullet}}\(\mathbb{S}^{1},F\)}=\left|\det\(1-A\)\right|^{-1}.
\end{align}

\begin{re}
	Equation \eqref{sa} is just Proposition \ref{prop} for  the 	rotation flow on  	$\mathbb{S}^{1}$. 
\end{re}

\begin{re}
	Since the flat vector bundle is not necessarily unimodular, i.e., 
	$|\det(A)|$ is 
	not necessarily equal to $1$,  the choice of the 
	orientation on $\bbS^{1}$ is very important.  
\end{re}

\subsection{A fusion principle}\label{Fp}
Let $Y_{1}\subset Y_{2}\subset Y_{3}$ be three compact smooth 
manifolds with boundaries of 
the same dimension such that $Y_{1}\subset \mathring{Y_{2}}$ and 
$Y_{2}\subset \mathring{Y_{3}}$. Let $F$ be a flat vector bundle on 
$Y_{3}$ of rank $r$. Denote by $H^{{\scriptscriptstyle\bullet}}(Y_{3},Y_{2},F)$, 
$H^{{\scriptscriptstyle\bullet}}(Y_{3},F)$, \ldots, the corresponding relative or absolute 
cohomologies  with coefficients in  $F$. 

As in \cite[(0.16)]{BruningMa_2012}, we have a long exact sequence 
\begin{align}\label{Les}
\cdots\to 	H^{i}(Y_{3},Y_{2},F)\to H^{i}(Y_{3},Y_{1},F)\to 
	H^{i}(Y_{2},Y_{1},F)\to H^{i+1}(Y_{3},Y_{2},F) \to  \cdots.
\end{align}
By \eqref{detCdetH} and \eqref{Les}, we get an isomorphism of lines 
\begin{align}\label{Maja3}
	f_{21,32	}:\det H^{{\scriptscriptstyle\bullet}}(Y_{2},Y_{1},F)\otimes \det 		H^{{\scriptscriptstyle\bullet}}(Y_{3},Y_{2},F)\simeq \det H^{{\scriptscriptstyle\bullet}}(Y_{3},Y_{1},F).
\end{align}
Using the other triples  $(\varnothing,Y_{1},Y_{2})$ and 
$(\varnothing,Y_{1},Y_{3})$, we get similar isomorphisms
\begin{align}\label{Maja4}
	\begin{aligned}
f_{1,21}:	\det H^{{\scriptscriptstyle\bullet}}(Y_{1},F)\otimes \det 	
	H^{{\scriptscriptstyle\bullet}}(Y_{2},Y_{1},F)\simeq \det H^{{\scriptscriptstyle\bullet}}(Y_{2},F),\\
f_{1,31}:	\det H^{{\scriptscriptstyle\bullet}}(Y_{1},F)\otimes \det 	
	H^{{\scriptscriptstyle\bullet}}(Y_{3},Y_{1},F)\simeq \det H^{{\scriptscriptstyle\bullet}}(Y_{3},F). 
\end{aligned}
\end{align}
By \eqref{Maja3} and \eqref{Maja4}, we see that 
$f_{2,32}\circ\(f_{1,21}\otimes\mathrm{id}\) $ and $f_{1,31}\circ \(\mathrm{id} \otimes f_{21,32}\)$
define  two isomorphisms
\begin{align}\label{FYYY}
	\det H^{{\scriptscriptstyle\bullet}}(Y_{1},F)\otimes \det H^{{\scriptscriptstyle\bullet}}(Y_{2},Y_{1},F)\otimes \det 	
	H^{{\scriptscriptstyle\bullet}}(Y_{3},Y_{2},F)	\simeq  \det H^{{\scriptscriptstyle\bullet}}(Y_{3},F).
\end{align}

\begin{prop}\label{propFp}
	There is $\mu=1$ or $-1$\footnote{See \cite[Remark 1.2]{BGS1} or 
	\cite{KM76} for the detail about the sign.  } such that 
	\begin{align}\label{mabu}
f_{2,32}	\circ\(f_{1,21}\otimes\mathrm{id}\) = \mu \times	
	f_{1,31}\circ \(\mathrm{id} \otimes f_{21,32}\). 
	\end{align}
\end{prop}
\begin{proof}
	As in \cite[(0.15)]{BruningMa_2012}, let us take a smooth triangulation 
	of $Y_{3}$ such that it induces also smooth triangulations on 
	$Y_{1}$ and $Y_{2}$. Denote by $(C^{{\scriptscriptstyle\bullet}}(Y_{1},F),d), 
	(C^{{\scriptscriptstyle\bullet}}(Y_{2},Y_{1},F),d)$, \ldots, the complexes of simplicial 
cochains with  coefficients in $F$. Then we have an exact sequence of 
complexes
\begin{align}\label{eq111}
	0\to (C^{{\scriptscriptstyle\bullet}}(Y_{2},Y_{1},F),d)\to (C^{{\scriptscriptstyle\bullet}}(Y_{2},F),d)\to 
	(C^{{\scriptscriptstyle\bullet}}(Y_{1},F),d)\to 0.
\end{align}
By \eqref{detCdetH} and \eqref{eq111}, we get an isomorphism of lines 
\begin{align}
	f^{C}_{1,21}: \det C^{{\scriptscriptstyle\bullet}}(Y_{1},F)\otimes \det 
	C^{{\scriptscriptstyle\bullet}}(Y_{2},Y_{1},F)\to \det C^{{\scriptscriptstyle\bullet}}(Y_{2},F).
\end{align}
We can define $f^{C}_{2,32}$, $f^{C}_{1,31}$ and $f^{C}_{21,32}$ in a 
similar way. By an easy calculation, there is $\mu=1$ or $-1$ such that 
\begin{align}\label{mabuC}
f_{2,32}^{C}	\circ\(f_{1,21}^{C}\otimes\mathrm{id}\) = \mu	\times 
	f_{1,31}^{C}\circ \(\mathrm{id} \otimes f_{21,32}^{C}\). 
	\end{align}

	By \eqref{Maja1} and \eqref{Maja11}, there is $\mu=1$ or $-1$ such that the  diagram 
	\begin{align}\label{eqBGSMA}
		\begin{aligned}
			\xymatrix{
	\det C^{{\scriptscriptstyle\bullet}}(Y_{1},F)\otimes \det C^{{\scriptscriptstyle\bullet}}(Y_{2},Y_{1},F)
		\ar[d]^{\tau_{C^{{\scriptscriptstyle\bullet}}(Y_{1},F)}\otimes\tau_{C^{{\scriptscriptstyle\bullet}}(Y_{2},Y_{1},F)}}
		\ar[r]^-{f^{C}_{1,21}} &\det C^{{\scriptscriptstyle\bullet}}(Y_{2},F)
\ar[d]^{\mu\tau_{C^{{\scriptscriptstyle\bullet}}(Y_{2},F)}}
\\
		\det H^{{\scriptscriptstyle\bullet}}(Y_{1},F)\otimes \det 
		H^{{\scriptscriptstyle\bullet}}(Y_{2},Y_{1},F)\ar[r]^-{f_{1,21}}         &\det H^{{\scriptscriptstyle\bullet}}(Y_{2},F)}	
	\end{aligned}
	\end{align}
	commutes. Tensoring each vertical line  of \eqref{eqBGSMA} by the 
	isomorphism 
	\begin{align}
	\tau_{C^{{\scriptscriptstyle\bullet}}(Y_{3},Y_{2},F)}:\det C^{{\scriptscriptstyle\bullet}}(Y_{3},Y_{2},F)\simeq 
	\det H^{{\scriptscriptstyle\bullet}}(Y_{3},Y_{2},F),
	\end{align}
	and using \eqref{eqBGSMA} again for the pair  $(Y_{2},Y_{3})$, we 
	find that 
	there is $\mu=1$ or $-1$ such that 	the diagram \begin{align}\label{eqDjht}
		\begin{aligned}
			\xymatrixcolsep{6pc}\xymatrix{
	\det C^{{\scriptscriptstyle\bullet}}(Y_{1},F)\otimes \det 
	C^{{\scriptscriptstyle\bullet}}(Y_{2},Y_{1},F)\otimes \det C^{{\scriptscriptstyle\bullet}}(Y_{3},Y_{2},F)
		\ar[d]^{\tau_{C^{{\scriptscriptstyle\bullet}}(Y_{1},F)}\otimes\tau_{C^{{\scriptscriptstyle\bullet}}(Y_{2},Y_{1},F)}\otimes\tau_{C^{{\scriptscriptstyle\bullet}}(Y_{3},Y_{2},F)}}
		\ar[r]^-{f^{C}_{2,32}	\circ(f^{C}_{1,21}\otimes\mathrm{id}) 
		} &\det C^{{\scriptscriptstyle\bullet}}(Y_{3},F)
\ar[d]^{\mu\tau_{C^{{\scriptscriptstyle\bullet}}(Y_{3},F)}}
\\
		\det H^{{\scriptscriptstyle\bullet}}(Y_{1},F)\otimes \det 
		H^{{\scriptscriptstyle\bullet}}(Y_{2},Y_{1},F)\otimes \det 
		H^{{\scriptscriptstyle\bullet}}(Y_{3},Y_{2},F)\ar[r]^-{f_{2,32}	\circ\(f_{1,21}\otimes\mathrm{id}\) }         
		&\det H^{{\scriptscriptstyle\bullet}}(Y_{3},F)}	
	\end{aligned}
	\end{align}
	commutes. 
 In \eqref{eqDjht}, if we replace the
	horizontal morphismes  by $f_{1,31}^{C}\circ \(\mathrm{id} \otimes 
	f_{21,32}^{C}\)$ and  $f_{1,31}\circ \(\mathrm{id} \otimes 
	f_{21,32}\)$, the corresponding  diagram still commutes. By \eqref{mabuC}, we get 
	\eqref{mabu}. 
\end{proof}

\section{Milnor metric}\label{S1M}
This section is organised as follows. In Sections \ref{sMS} and
\ref{sDZf}, we recall
the definitions of Morse-Smale flow and the associated Ruelle dynamical zeta
function. In Section \ref{sMSMF}, we recall  some results due to Franks
\cite{Franks1979,Franks1982} on the construction of  a new
gradient  flow  by destroying  the closed orbits of the original Morse-Smale flow.  In Section
\ref{sSF}, using the Smale filtration, we introduce the Milnor 
metric. In Section  \ref{sCMM}, we establish a comparison  formula for the two Milnor
metrics,  one is associated with the Morse-Smale flow and the other is
associated with the gradient  flow constructed by Franks.

We refer  the reader to  the classical textbook of Palis and de 
Melo \cite{Palis82} for the basic notion on dynamical system.

\subsection{Morse-Smale flow}\label{sMS}
Let $X$ be a connected closed smooth manifold of dimension $m$. Let $V$ be a
 vector field on $X$. Consider the differential equation 
 \begin{align}\label{eqDiffV}
	\frac{dx}{dt}=V(x).
\end{align} 
 Equation \eqref{eqDiffV} defines a group of differomorphism 
 $\(\phi_{t}\)_{t\in \bR}$ of $X$.

If $x\in X$,  an orbit of $x$ is defined by the image $t\in \bR\to 
\phi_{t}(x)\in X$.  We call $x\in X$ is a fixed point, if its  orbit reduces to a point, i.e,  for all $t\in \bR$,  
\begin{align}
	\phi_{t}(x)=x. 
\end{align} 
Clearly, $x\in X$ is a fixed point if and only if $V(x)=0$. 
We call an orbit is closed if it is diffeomorphic to $\bbS^{1}$. 
Denote by $A$ the set of fixed points and by $B$ the set of closed orbits.

\begin{defin}
A fixed point $x\in X$ of the flow $\phi_{\cdot}$ is called 
hyperbolic if there is a $\phi_{t}$-invariant splitting
\begin{align}
	T_{x}X= E^{u}_{x}\oplus E^{s}_{x},
\end{align}
and there exist $C>0,\theta>0$ and a Riemannian metric $g^{TX}$ on $X$
such that for  $v\in E^{u}_{x}$,  $v'\in E^{s}_{x}$, and $t>0$, we have
\begin{align}\label{eqMS}
&\left|\phi_{- t,*}v\right|\l
Ce^{-\theta t}\left|v\right|,&\left|\phi_{ t,*}v'\right|\l
Ce^{-\theta t}\left|v'\right|.
\end{align}
The unstable and stable manifolds of the  hyperbolic fixed point $x$ are defined by
\begin{align}
	&W^{u}_{x}=\Big\{y\in X: \lim_{t\to -
	\infty}d_{X}(\phi_{t}(y),x)=0\Big\},&W^{s}_{x}= \Big\{y\in X: \lim_{t\to
	\infty}d_X(\phi_{t}(y),x)=0 \Big\},
\end{align}
where $d_{X}$ denotes the Riemannian distance on $(X,g^{TX})$. The index $\ind(x)\in \mathbf{N}$ of $x$ is defined by 
\begin{align}\label{eqindxEu}
	\ind(x)=\dim E^{u}_{x}.
\end{align}
\end{defin}

Note that if $V=-\nabla f$ is a negative gradient of a Morse function 
$f$
with respect to some Riemannian metric, then the index $\ind(x)$ of 
the critical point $x$ is just the Morse index of $f$ at $x$. 

\begin{defin}
A  closed orbit $\gamma$ of the flow $\phi_{\cdot}$ is called 
hyperbolic,
if there is  a $\phi_{t}$-invariant continuous  splitting 
\begin{align}\label{eq:AF1}
TX|_{\gamma}=\mathbf{R}V\oplus E^{u}_{\gamma}\oplus E^{s}_{\gamma},
\end{align}
of $C^{0}$-vector 
bundles over $\gamma$ such that \eqref{eqMS} holds.  The associated unstable and
stable manifolds are defined by
\begin{align}
\begin{aligned}
	W^{u}_{\gamma}=\bigcup_{x\in \gamma} \Big\{y\in X: \lim_{t\to
	-\infty}d_{X}(\phi_{t}(y),\phi_{t}(x))=0 \Big\},\\ W^{s}_{\gamma}=\bigcup_{x\in \gamma} \Big\{y\in X: \lim_{t\to
	+\infty}d_{X}(\phi_{t}(y),\phi_{t}(x))=0 \Big\}.
	\end{aligned}
\end{align}
The index $\ind(\gamma)\in \mathbf{N}$ of $\gamma$ is defined by 
\begin{align}\label{indexg}
	\ind(\gamma)=\rk E^{u}_{\gamma}.
	\end{align}
\end{defin}

\begin{defin}
	The nonwandering set of $\phi_{\cdot}$ is defined by 
	\begin{align}
		\Big\{x\in X: \hbox{ $\forall$ open neighborhood $U$ of $x$,  
		$\forall\, T>0$, we have }  U\cap\bigcup_{t\g T} \phi_{t}(U)\neq 
		\varnothing \Big\}.
	\end{align}
\end{defin}
Clearly, $A\cup \bigcup_{\gamma\in B}\gamma$ is contained in the nonwandering set.


\begin{defin}
	A vector field $V$ or a flow $\phi_{\cdot}$ is called Morse-Smale if
	\begin{itemize}
		\item the sets $A$ and $B$ are finite and contain only hyperbolic elements; 
       \item the nonwandering set of
$\phi_{\cdot}$ is equal to
$	A\cup \bigcup_{\gamma\in B}\gamma$;
     \item the stable  and unstable manifold of any critical
	 element in $A\coprod B$  intersect transversally.
	\end{itemize}
\end{defin}

In the sequel , we assume
that $V$ is a
Morse-Smale vector field.

\subsection{Ruelle dynamical zeta function}\label{sDZf}
 For $\gamma\in B$, denote by $\ell_{\gamma}\in \bR^{*}_{+}$ its
 minimal period. 
%
 A closed orbit $\gamma\in B$
 is called untwist if $E^{u}_{\gamma}$ is orientable along $\gamma$, and is
 called twist otherwise. Put
 \begin{align}\label{Delta}
	 \Delta(\gamma)=\left\{
	 \begin{array}{cl}
		 1,&\hbox{if } \gamma \hbox{ is untwist},\\
		 -1,&\hbox{if } \gamma \hbox{ is twist}.
	\end{array}
	 \right.
 \end{align}

 Let $\rho:\pi_{1}(X)\to \mathrm{GL}_{r}(\bC)$ be a representation of the fundamental group of $X$. If
 $\gamma\in B$, denote by $\rho(\gamma)$
 the holonomy along $\gamma$. Clearly, $\rho(\gamma)$ is well-defined up to a conjugation.

\begin{defin} The twist Ruelle dynamical zeta function is a meromorphic
	function on $\bC$ defined for $s\in \bC$ by
\begin{align}\label{rrs}
	R_{\rho}(s)=\prod_{\gamma\in
	B}\det\(1-\Delta(\gamma)\rho(\gamma)e^{-s
	\ell_\gamma}\)^{(-1)^{\ind(\gamma)}}.
\end{align}
\end{defin}

\subsection{Franks' Morse function}\label{sMSMF}
 We follow \cite[Section 1]{Franks1979}. Let $\mathbb{D}^{r}$ be the
$r$-dimensional open unit ball of center $0\in
\mathbf{R}^{r}$.
 A fixed point $x\in A$ of index $p$ is said to be 
 of standard form if there 
 is a system of 
 coordinates $(y_{1},\ldots, y_{m})\in \mathbb{D}^{m}$ on a neighborhood
 of $x$ such that $x$ is represented by $0$ and
 \begin{align}
	 V=\sum_{i=1}^{p}y_{i}\frac{\p}{\p
	 y_{i}}-\sum_{i=p+1}^{m}y_{i}\frac{\p}{\p y_{i}}.
\end{align}

For closed orbits we must distinguish the following four cases  in
establishing the standard forms. Assume $\gamma\in B$ such that 
$\ind(\gamma)=p$. 
\begin{description}
	\item[Case (1)] Suppose that $TX|_{\gamma}$ is orientable and that  $\gamma$ is
	untwist. In this case, $\gamma$ is said to be of standard form, if there 	
	is a  system of coordinates $(t,y_{1},\ldots,
 y_{m-1})\in \mathbb{S}^{1}\times \mathbb{D}^{m-1}$  
 on a tubular neighborhood
	$U_{\gamma}$	of $\gamma$ such that $\gamma$ is represented by
 $(t,0)\in  \mathbb{S}^{1}\times \mathbb{D}^{m-1}$ and
\begin{align}\label{eqV}
	 V=\frac{\p}{\p t}+\sum_{i=1}^{p}y_{i}\frac{\p}{\p
	 y_{i}}-\sum_{i=p+1}^{m-1}y_{i}\frac{\p}{\p y_{i}}.
\end{align}
	\item[Case (2)] Suppose that $TX|_{{\gamma}}$ is orientable and 
	that   $\gamma$ is
	twist. In this case, $\gamma$ is said to be of standard form, if
	 $U_{\gamma}$ and $V$ can be obtained
	from Case (1) by  the identification
	\begin{align}\label{enid}
	(t,
	x_{1},\ldots,x_{m-1})\sim 	(t+1/2,
	-x_{1},x_{2},\ldots, x_{p}, -x_{p+1}, x_{p+2},\ldots,x_{m-1}).
	\end{align}	
	\item[Case (3)]  Suppose that   $TX|_{{\gamma}}$ is not orientable  and
 that  $\gamma$ is
	untwist. In this case, $\gamma$ is said to be of standard form, if
	 $U_{\gamma}$ and $V$ can be obtained
	from Case (1) by  the identification
	\begin{align}
	(t,
	x_{1},\ldots,x_{m-1})\sim 	(t+1/2,
	x_{1},\cdots,x_{p},-x_{p+1},x_{p+1},\ldots,x_{m-1}).	\end{align}	
	\item[Case (4)]  Suppose that  $TX|_{{\gamma}}$ is not orientable 
	and that 	 $\gamma$ is
	twist. In this case,  $\gamma$ is said to be of standard form, if
	 $U_{\gamma}$ and $V$ can be obtained
	from Case (1) by  the identification
	\begin{align}
		(t,
	x_{1},\ldots,x_{m-1})\sim (t+1/2,
	-x_{1},x_{2},\cdots,x_{m-1}).
	\end{align}
\end{description}
Note that in \cite[Section 1]{Franks1979} the author assumed that $X$ is orientable, so only the first two cases appear.

The following three propositions are  \cite[Proposition 1.6,
Theorem 2.2, Proposition 5.1]{Franks1979}.
Their proofs can be generalised to the non orientable case
with some evident modifications. We omit the details.


\begin{prop}\label{propV01}
	 For any Morse-Smale vector field $V$, there is a smooth family of Morse-Smale
	 vector fields $(V_{\ell})_{0\l \ell\l 1}$ such that $V_{0}=V$ and that the critical
	 elements of
	 $V_{1}$ are all of 
standard forms and are precisely the same as the critical
	 elements of $V$.  Moreover, $V_0$ and $V_1$ are topologically 
	 conjugated, i.e., there is a homeomorphism carrying the orbits 
	 of $V_0$ to those of $V_{1}$ and preserving their orientations.
\end{prop}

\begin{re}
	The second part of  Proposition \ref{propV01} is a consequence of 
	the Structural stability  of the Morse-Smale flow \cite{Palis68,PalisSmale70}. 
\end{re}

\begin{re}\label{reVl}
	Following the proof of \cite[Proposition  	1.6]{Franks1979} 
	given by Franks, we can choose the family $(V_{\ell})_{0\l 
	\ell 		\l1}$ such that the critical elements are preserved 
	under the deformation. However, in the proof of our main result 
	Theorem \ref{thm1} given in Section \ref{S2}, we need only 
	choose a family such that all the set of the fixed points of 
	$V_{\ell}$ are in a small neighbourhood of the set of the fixed points of $V$. 
\end{re}

%


The relation between the Morse-Smale flow and the  Morse function is
summarised in the following two propositions. The first one is due to 
Smale \cite[Theorem B]{S61}.

\begin{prop}
	If $V$ is a Morse-Smale vector field whose flow has fixed points
	in standard form and no closed orbits, then $V$ is a  negative gradient of
a certain Morse function with respect to some Riemannian metric.
\end{prop}

To state the following  proposition, let us introduce some notation. 
For $x,y\in A$  such that $\ind(y)=\ind(x)-1$, then $W^{u}_{x}\cap W^{s}_{y}$
consists of a finite set $\Gamma(x,y)=\{a_{\cdot}\}$ of integral curves of
$V$ such that $a_{-\infty}=x$ and $a_{\infty}=y$ along
which $W^{u}_{x}$ and $W^{s}_{y}$ intersect transversally. Let us fix 
an orientation on each $W^{u}_{x}$ with $x\in A$. Define
$n(a)=\pm1$ as in  \cite[(1.28)]{BZ92}, whose definition does not 
require the manifold to be orientable. 

\begin{prop}\label{propfV}
	For some small neighborhood $U=\cup_{\gamma\in B}U_{\gamma}$ of 
	closed orbits $\cup_{\gamma\in B}\gamma$, there is
	a 	Morse function $f$ on $X$ whose gradient vector field $\nabla f$
	with respect to a certain Riemannian metric is
	Morse-Smale, such that 
	\begin{itemize}
		\item on $X\backslash U$, we have 
	\begin{align}\label{eqdf=V}
		-\nabla f=V,
	\end{align}
	\item on each $U_{\gamma}$, the Morse function  $f$ has only two  critical points $x_{\gamma},x'_{\gamma}$ of
 index $\ind(\gamma)+1$ and $\ind(\gamma)$ respectively.
	\end{itemize}
	  Also, $\Gamma(x_{\gamma},x'_{\gamma})$ consists of two
 integral curves  $a_{\gamma},a'_{\gamma}$ (see Figure \ref{fig:1})
 such that their composition 
 $a_{\gamma}\circ  (a'_{\gamma})^{-1}$ and the
 closed orbit $\gamma$ lie 
 in the same freely  homotopy class of
 loops on $X$ and  that\footnote{This 
	 requires a choice of the orientations on the unstable manifolds of 
	 $x_{\gamma},x_{\gamma}'$. Such choice is irrelevant.}
 \begin{align}\label{nana}
	 n(a_{\gamma})n(a'_{\gamma})=-\Delta(\gamma).
\end{align}
\end{prop}

\begin{figure}[htbp] 
\centering
\includegraphics[width=3.5in]{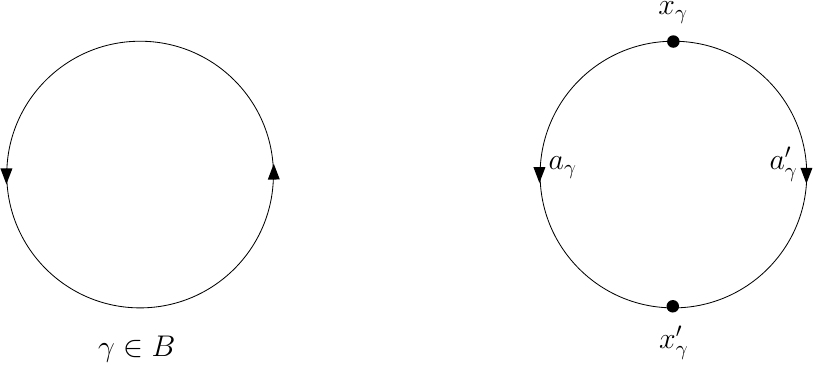} 
\caption{A closed orbit  and  integral curves }\label{fig:1} 
\end{figure}

\begin{re}\label{reK}We recall the essential step in the
	construction of the gradient $\nabla f$ given by Franks \cite[Proposition
	8.5]{Franks1982}. For simplicity, assume that	there is only one
	closed orbit 	$\gamma$ and it is of index $p$ and in standard 
	form of Case (1).
For $\delta>0$ small enough, let $\rho\in 
C^{\infty}_{c}(\mathbb{D}^{m-1},[0,1])$
be a cutoff function, which is equal to
$1$ on $|y|<\delta$ and to $0$ on $|y|>2\delta$. Put
\begin{align}
		V_{1}=\left\{
		\begin{array}{lr}
		\big((1-\rho)+\rho\sin\(2\pi t\)\big)\frac{\p}{\p t}+\sum_{i=1}^{p}y_{i}\frac{\p}{\p
	 y_{i}}-\sum_{i=p+1}^{m-1}y_{i}\frac{\p}{\p y_{i}}, &\hbox{ in
	 } 	 U_{\gamma},\\
	 V,& \hbox{ outside } U_{\gamma}.
		\end{array}\right.
	\end{align}
It is easy to see that the
	nonwandering set of $V_{1}$ in $U_{\gamma}$ consists of two
	points $(1,0),(\frac{1}{2},0)\in
	\mathbb{S}^{1}\times \mathbb{D}^{m-1}$. 
	Then, by a small perturbation on $V_{1}$,  we  get a Morse-Smale gradient vector
	field  $-\nabla f$ which
	has the desired transversality and other properties. 

We remark that by the above construction, we can find  a	family of vector fields $(V_{\e})_{0\l\e\l1}$ connecting $V$ and
	$-\nabla f$ such that near $\{1/4\}\times \mathbb{D}^{m-1}$, for
	any $ \e\in [0,1]$, $V_{\e}$ does not vanish.
	Similar remark holds for $\gamma$ in standard forms of Cases (2)-(4).
	In Section \ref{s33}, we will use this fact to simplify the proof of our main theorem.
\end{re}

\subsection{Smale filtration and Milnor metric}\label{sSF}
 Following \cite[Definition 9.10]{Franks1982},
let 
\begin{align}\label{eqSmale}
	\varnothing=X^{0}\subset X^{1}\subset\cdots \subset X^{N}=X
\end{align}	
be a Smale filtration on $X$ associated with $V$. Note that each  
$X^{p}\subset X$ is a submanifold with boundary, and can be constructed by the sublevel set  of a smooth Lyapunov function. Also, we 
have 
\begin{itemize}
	\item on each $\p X^{p}$, $V$ does not vanish and points toward the inside of $X^{p}$;
	\item  there is only one critical element $c\in A\coprod 
	B$ contained in $X^{p+1}\backslash 
	X^{p}$ and 
	\begin{align}\{c\}=\bigcap_{t\in \bR}\phi_{t}\(X^{p+1}\backslash 
	X^{p}\).
	\end{align} 
\end{itemize}

Let $(F,\nabla^{F})$ be a flat vector bundle on $X$ induced by the representation 
$\rho$. 
Let $H^{{\scriptscriptstyle\bullet}}(X,F)$ be the cohomology of the sheaf of locally constant
sections of $F$. Put
\begin{align}\label{eqlambda}
	\lambda=\det H^{{\scriptscriptstyle\bullet}}(X,F).
\end{align}
We use the notation in Section \ref{Fp}. By \eqref{FYYY}, we get an isomorphism 
 \begin{align}\label{eqSS}
 	\sigma_{V}:\bigotimes_{p=0}^{N-1}\det 
	H^{{\scriptscriptstyle\bullet}}(X^{p+1},X^{p},F)\simeq \lambda.
 \end{align}
By Proposition \ref{propFp}, up to a sign, the 
morphism $\sigma_{V}$ does not depend on the way that the 
cohomologies are fused. 

By \cite[Theorem 9.11]{Franks1982} (see also \cite[Section 2]{MorgadoMorseSmale}), if the critical element  $c\in 
X^{p+1}\backslash X^{p}$ is a fixed 
point, then 
\begin{align}\label{E1x}
	H^{q}\(X^{p+1},X^{p},F\)=\left\{ \begin{array}{ll}
	F_{c},& q=\ind (c),\\
	0,&\hbox{otherwise},
\end{array}\right.
\end{align}
and if  the critical element  $c\in 
X^{p+1}\backslash X^{p}$ is a closed orbit, then 
\begin{align}\label{E1g}
	H^{q}\(X^{p+1},X^{p},F\)=\left\{ \begin{array}{ll}
	H^{q-\ind (c)}\big(c,o(E^{u}_{c})\otimes F|_{c}\big),& q-\ind( c)=0 
	\hbox{ or }1,\\
	0,&\hbox{otherwise,}
\end{array}\right.
\end{align}
where $o(E_{c}^{u})$ is the orientation line bundle of $E^{u}_{c}$ along the 
closed orbit $c$.

We equip $\det 
	H^{{\scriptscriptstyle\bullet}}\(\gamma,o(E^{u}_{\gamma})\otimes 
	F|_{\gamma}\)$ with the metric $\|\cdot\|^{2}_{\det 
	H^{{\scriptscriptstyle\bullet}}\(\gamma,o(E^{u}_{\gamma})\otimes 
	F|_{\gamma}\)}$ defined  in \eqref{nsa}. 
Let $g^{F}$ be a Hermitian metric on $F$. By \eqref{eqSS}-\eqref{E1g}, the
restriction  $g^{F}|_{A}$ 
and  the metric $\|\cdot\|^{2}_{\det 
	H^{{\scriptscriptstyle\bullet}}\(\gamma,o(E^{u}_{\gamma})\otimes 
	F|_{\gamma}\)}$ induce a metric 
%
$\|\cdot\|^{M,2}_{\lambda,V}$ on
$\lambda$. By Proposition \ref{propFp},
this metric  does not depend on the way that the 
cohomologies are fused. 

\begin{defin}
	The  metric $\|\cdot\|^{M,2}_{\lambda,V}$ on
$\lambda$ is called the Milnor metric associated with 
$V$.
\end{defin}

\begin{re}
	If $V=-\nabla f$ is a negative  gradient of a Morse function $f$, then 
	$\|\cdot\|^{M,2}_{\lambda,V}$ coincides with the one constructed by 
	Bismut-Zhang \cite[Definition 1.9]{BZ92}. 
	In fact, there is a small difference with Bismut-Zhang's 
	construction, where  they used a filtration \cite[(1.37)]{BZ92} induced	
	by sublevel sets of a nice 	Morse function.  Using Proposition 
	\ref{propFp}, we can deduce  that the two constructions coincide. 
\end{re}

\begin{re}\label{reDMA}
	For two topologically conjugated Morse-Smale vector fields whose 
	critical elements coincide, we can choose the same Smale filtration. From our construction,  the corresponding Milnor metrics coincide. 
\end{re}

\begin{re}
	The Milnor metric for 
	general Morse-Smale flow 	does not depend on the Smale 
filtration \eqref{eqSmale}. We will not give a direct proof  since it is a consequence of our main Theorem \ref{thm1}.
\end{re}

Let us restate and prove Proposition \ref{prop}. 

\begin{prop}\label{Prop}
	If $V$ does not have any fixed points, and 
if none of $\Delta(\gamma)$ is an  eigenvalue of $\rho(\gamma)$, then 
$H^{{\scriptscriptstyle\bullet}}(X,F)=0$, and the norm of the canonical section 
$\mathbf{1}\in \bC=\det H^{{\scriptscriptstyle\bullet}}(X,F)$ is given by 
\begin{align}\label{eqReMi}
	\|\mathbf{1}\|^{M}_{\lambda,V}=|R_{\rho}(0)|^{-1}. 
\end{align}
\end{prop}
\begin{proof}
	For $\gamma\in B$, the holonomy of $o(E_{\gamma}^{u})\otimes 
	F|_{\gamma}$ 
	along $\gamma$ is $\Delta(\gamma)\rho(\gamma)$. By our 
	assumption, 
\begin{align}\label{eqDMA}
H^{{\scriptscriptstyle\bullet}}(\gamma,o(E_{\gamma}^{u})\otimes 
F_{\gamma})=0.
\end{align} 
	By \eqref{Les}, \eqref{E1x}, \eqref{E1g}, and \eqref{eqDMA}, we can deduce  that  $H^{{\scriptscriptstyle\bullet}}(X,F)=0$. By 
	\eqref{sa}, \eqref{rrs}, and \eqref{eqSS}, we get 
	\eqref{eqReMi}. 
\end{proof}

\subsection{A comparison formula for  Milnor metrics}\label{sCMM}
In this section, we assume that all the critical elements of
$V$ are in standard forms, and that $f$ is chosen as in
Proposition 	\ref{propfV}.  

Let $\det \tau({a'_{\gamma}})\in \det
		F_{x'_{\gamma}}\otimes \(\det
		F_{x_{\gamma}}\)^{-1}$ be the canonical element induced by the
	parallel transport  with respect to the flat connection	along the 
	integral curve $a_{\gamma}'$ (see Figure \ref{fig:1}). Let $\|\cdot\|^{2}_{\det
		F_{x'_{\gamma}}\otimes \(\det
		F_{x_{\gamma}}\)^{-1}}$ be the metric on $\det
		F_{x'_{\gamma}}\otimes \(\det
		F_{x_{\gamma}}\)^{-1}$ induced by $g^{F}_{x_{\gamma}}$ and
		$g^{F}_{x'_{\gamma}}$. 

\begin{prop}\label{propMM}
	The following identity holds,	
	\begin{align}\label{eqMM}
		\log\(\frac{\|\cdot\|^{M,2}_{\lambda,-\nabla
		f}}{\|\cdot\|^{M,2}_{\lambda,V}}\)=\sum_{\gamma\in
		B}(-1)^{\ind(\gamma)}\log
		\left\|\det\tau(a_{\gamma}')\right\|^{2}_{\det
		F_{x'_{\gamma}}\otimes \(\det
		F_{x_{\gamma}}\)^{-1}}.
	\end{align}
\end{prop}
\begin{proof}
	We refine the filtration \eqref{eqSmale} by the new critical 
	points of $f$. By Propositions \ref{propFp} and \ref{propfV}, and by 
	\eqref{eqSS}, the following diagram commutes
	 \begin{align}\label{eqDia}
		\begin{aligned}
			\xymatrix{
	\bigotimes_{x\in A}  \big(\!\det F_{x}\big)^{ (-1)^{\ind (x)}} 
		\bigotimes_{\gamma\in B} \left\{\det
		F_{x'_{\gamma}}\otimes \(\det
		F_{x_{\gamma}}\)^{-1}\right\}^{(-1)^{\ind (\gamma)}}
		\ar[d] 
		\ar[r]^-{\sigma_{-\nabla f}} &\lambda\ar[d] 
\\
		\bigotimes_{x\in A}\!\big(\!\det 
F_{x}\big)^{(-1)^{\ind (x)}} \!
	\bigotimes_{\gamma \in B}\Big\{\!\det 
	H^{{\scriptscriptstyle\bullet}}\(\gamma,o(E^{u}_{\gamma})\otimes 
	F|_{\gamma}\)\!\Big\}^{(-1)^{\ind 
	(\gamma)}}\ar[r]^-{\sigma_V}         &\lambda}	
	\end{aligned}
	\end{align}
%
%
where the first vertical arrow  is induced by \eqref{eqsigmaa} and 
the second vertical arrow is a multiplication  by $\pm 1$. %
	The Milnor metric $\|\cdot\|_{\lambda,-\nabla f}^{M,2}$ is
	obtained   from the metric $g^{F}|_{A \cup 
	\{x_{\gamma},x'_{\gamma}:\gamma\in B\}}$ via $\sigma_{-\nabla f}$. 
%
%
	By \eqref{eqDia}, we get  
	\eqref{eqMM}.
\end{proof}

\section{An extension of  Bismut-Zhang's Theorem to Morse-Smale 
flow}\label{S2}
This section is organised as follows. In Sections
\ref{sec:Bere}-\ref{sMa},  following \cite{BZ92}, we recall the 
constructions of the Berezin integral, the  Mathai-Quillen current, and the Ray-Singer metric. 
In Section \ref{s33}, we restate and prove our main theorem.  
\subsection{Berezin integral}\label{sec:Bere}
Let $E$ be a real Euclidean space of dimension $n$ with
the scalar product $\<,\>$, and let $W$ be a real vector space of 
finite dimension. We use the supersymmetric formalism of Quillen 
\cite{Quillensuper}. Denote by $\widehat{\otimes}$ the tensor product of super algebras.

Suppose temporarily that $E$ is oriented and that $e_1,\ldots, e_n$ 
is an oriented orthonormal  basis of $E$. Let $e^1,\ldots,e^n$ be the corresponding dual basis of $E^*$. We define $\int^B$ to be the linear map from $\Lambda^{\scriptscriptstyle\bullet}(W^*)\widehat{\otimes}\Lambda^{\scriptscriptstyle\bullet}(E^*)$ into $\Lambda^{\scriptscriptstyle\bullet}(W^*)$, such that if $\alpha\in \Lambda^{\scriptscriptstyle\bullet}(W^*), \beta\in \Lambda^{\scriptscriptstyle\bullet}(E^*)$,
\begin{align}\label{eq:defBere}
& \int^B\alpha\beta=0, \hbox{ if } \deg \beta<n,&\int^B\alpha 
e^1\wedge\cdots\wedge e^n=\frac{(-1)^{n(n+1)/2}}{\pi^{\frac{n}{2}}}\alpha.
\end{align}
More generally, if $o(E)$ is the orientation line of $E$, then $\int^B$ defines a linear map from $\Lambda^{\scriptscriptstyle\bullet}(W^*)\widehat{\otimes}\Lambda^{\scriptscriptstyle\bullet}(E^*)$ into $\Lambda^{\scriptscriptstyle\bullet}(W^*)\widehat{\otimes} o(E)$, which is called a Berezin integral.

Let $A\in \End^{\mathrm{anti}}(E)$ be an antisymmetric endomorphism of $E$. We identify $A$ with
\begin{align}\label{eq:ALam}
 \dot{A}=\frac{1}{2}\sum_{1\l i,j\l n}\<e_i,Ae_j\>e^i\wedge e^j\in \Lambda^2(E^*).
\end{align}
By definition, the Pfaffian $\mathrm{Pf}[A]\in o(E)$ of $A$ is given by
\begin{align}\label{eq:PfA}
\mathrm{Pf}\[\frac{A}{2\pi}\]=\int^B\exp\(-\frac{\dot{A}}{2}\).
\end{align}
Clearly, $\mathrm{Pf}[A]$ vanishes if $n$ is odd.

\subsection{Mathai-Quillen formalism }\label{sMQ}
Recall that  $X$ is a connected closed smooth    manifold of dimension $m$. Let 
$E$ be a Euclidean vector bundle of rank $n$ on  $X$ with a Euclidean 
metric $g^{E}$ and a metric connection  $\nabla^{E}$. Let $R^{E}=\(\nabla^{E}\)^{2}\in 
\Omega^{2}(X,\End^{\rm anti}(E))$ be its curvature.  Denote by $o(E)$  the orientation line
bundle of $E$. 
The Euler form of 
$(E,\nabla^{E})$ is given by 
\begin{align}
	e\(E,\nabla^{E}\)=\mathrm{Pf}\[\frac{R^{E}}{2\pi}\] \in 
	\Omega^{n}(X,o(E)). 
\end{align}
Clearly, $e\(E,\nabla^{E}\)=0$,  if $n$ is odd.

Let $\cE$ be the total space of $E$, and let $\pi: \cE\to X$ be the
natural projection. We will use the formalism of the Berezin integral developed in 
Section \ref{sec:Bere} with $W = T\cE$. If $\omega$ is a smooth section of  
 $\Lambda^{{\scriptscriptstyle\bullet}}(T^{*}\cE)\otimes\pi^{*}\Lambda^{{\scriptscriptstyle\bullet}}(E^{*})$ 
 over $\cE$, then $\int^{B}\omega$ is a smooth section of $\Lambda^{{\scriptscriptstyle\bullet}} 
 (T^{*} \cE) 
 \otimes \pi^{*}o(E)$ over $\cE$.

Let $T^{V}\cE\subset T\cE$ be the vertical subbundle of $T\cE$, and 
let $T^{H}\cE\subset T\cE$ be the horizontal  subbundle of $T\cE$ 
with respect to $\nabla^{E}$, so that 
\begin{align}\label{splitting}
	T\cE=T^{H}\cE\oplus  T^{V}\cE. 
\end{align}
By the identification 
$T^{V}\cE\simeq \pi^{*}E$, 
the 
vertical projection with 
respect to the splitting \eqref{splitting} induces a section 
$P^{E}\in C^{\infty}(\cE,T^{*}\cE \otimes \pi^{*}E)$. Using the metric $g^{E}$, we identify $P^{E}$ with $\dot{P}^{E}\in C^{\infty}(\cE,T^{*}\cE \otimes \pi^{*}E^{*})$. 
Let $Y\in C^{\infty}(\cE,\pi^{*}E)$ be the tautological section. 
Write $\widehat{Y}$ the corresponding  section  in 
$C^{\infty}(\cE,\pi^{*}E^{*})$ induced by  $g^{E}$. 
Recall that $\dot{R}^{E}\in C^{\infty}(X,\Lambda^{2}(T^{*}X)\otimes \Lambda^{2}(E^{*}))$ is 
defined in \eqref{eq:ALam}.

\begin{defin}
	For 
$T\g 0$, set
\begin{align}
	A_{T}=\frac{1}{2}\pi^{*}\dot{R}^{E}+\sqrt{T}\dot{P}^{E}+T|Y|^{2}\in C^{\infty}\big(\cE,\Lambda^{{\scriptscriptstyle\bullet}}(T^{*}\cE)\otimes \pi^{*}\Lambda^{{\scriptscriptstyle\bullet}}(E^{*})\big). 
\end{align}
	Let $(\alpha_{T})_{T\g0}$ and $(\beta_{T})_{T>0}$ be  families  of forms on $\cE$ defined 
	by 
	\begin{align}\label{eqDang1}
		\begin{aligned}
		&\alpha_{T}=\int^{B}\exp\(-A_{T}\)\in 
		\Omega^{n}(\cE,\pi^{*}o(E)),\\
		&\beta_{T}=\int^{B}\frac{\widehat{Y}}{2\sqrt{T}}\exp\(-A_{T}\)\in \Omega^{n-1}(\cE,\pi^{*}o(E)). 
		\end{aligned}
	\end{align}
\end{defin}
Clearly,
\begin{align}
	\alpha_{0}=\pi^{*}e\(E,\nabla^{E}\).
\end{align}
Let us recall \cite[Theorems 3.4, 3.5, and 3.7]{BZ92}. 
\begin{thm}
For $T\g 0$, the form $\alpha_{T}$ is closed whose  cohomology class 
does not depend on $T$.
For $T>0$,  $\alpha_{T}$ represents the Thom 
class of $E$ so that $\pi_{*}\alpha_{T}=1$, and  we have 
\begin{align}
&\frac{\p \alpha_{T}}{\p T}=-d\beta_{T}. 
\end{align}
\end{thm}

We identify $X$ as a submanifold of $\cE$ by the zero section.  The normal bundle 
to  $X$ in $\cE$ is exactly $E$ and the conormal bundle is $E^{*}$.  
Let $\delta_{X}$ be the current on $\cE$ defined by  the integration  
on $X$. If $\mu$ is a smooth compactly supported form on $\cE$ with 
values in $\pi^{*}o(TX)$, then 
\begin{align}
	\int_{\cE}\mu\delta_{X}=\int_{X}\mu.
\end{align} 
For a current $v$ on $\cE$, denote by $\mathrm{WF}(v)\subset 
T^{*}\cE$ its wave front set \cite[Section 8.1]{HormanderBook1}.

\begin{thm}
Let $K\subset 
\cE$ be a compact subset of $\cE$. There is $C_{K}>0$ such that  for 
any $\mu\in 
\Omega^{{\scriptscriptstyle\bullet}}(\cE,\pi^{*}o(TX))$ whose support is contained in $K$ 
and for any $T\g 1$, we have
\begin{align}\label{eqabT}
&\left|	\int_{\cE}\mu(\alpha_{T}-\delta_{X})\right|\l 
\frac{C_{K}}{\sqrt{T}}\|\mu\|_{C^{1}},&\left|\int_{\cE}\mu \beta_{T}\right|\l 
\frac{C_{K}}{T^{3/2}}\|\mu\|_{C^{1}},
\end{align}	
where $\|\cdot\|_{C^1}$ denotes the $C^{1}$-norm. 
\end{thm}

 In view of \eqref{eqDang1} and \eqref{eqabT},
 \begin{align}
	 \psi\(E,\nabla^{E}\)=\int_{0}^{\infty}\beta_{T}dT,
	 \end{align}
 is a well-defined current of degree $n-1$ on $\cE$ with values in $\pi^{*}o(E)$.
 
\begin{thm}
The current $\psi\(E,\nabla^{E}\)$ is locally 
integrable such that 
\begin{align}\label{WFphi}
	\mathrm{WF} \big(\psi\(E,\nabla^{E}\)\big)\subset E^{*}.
\end{align}
  The following identity of currents
on $\cE$ holds,
\begin{align}
	d\psi\(E,\nabla^{E}\)=\pi^{*}e\(E,\nabla^{E}\)-\delta_{X}.
\end{align}
\end{thm}

 \begin{defin}
The  current  $\psi\(E,\nabla^{E}\)$ is called the Mathai-Quillen 
current. 
\end{defin}

\begin{re}
	The restriction of $\psi(E,\nabla^{E})$ to the
	sphere bundle  of $E$
	was first constructed in Mathai-Quillen \cite[Section 
	7]{MathaiQuillen}. 
	If $E=TX$, this restriction  coincides with the transgressed Euler
	class defined by Chern \cite{Chern2}.
\end{re}

Assume now $n\l m$. Let $s\in C^{\infty}(X,E)$ be a smooth section of $E$. Set 
\begin{align}\label{eqX0}
	X'=\{x\in X:s(x)=0\}. 
\end{align} 
Suppose that over $X'$, the differential of $s$ has maximal rank. By transversality, 
$X'$ is a smooth submanifold of $X$ of dimension $m-n$. Let $N_{X'/X}$ be the normal bundle to $X'$ in $X$. Using \cite[Theorem 
8.2.4]{HormanderBook1}, Bismut and Zhang have shown the following proposition in \cite[Remark 3.8]{BZ92}.

\begin{prop}\label{PropBZ38}
	The pull back currents $s^{*}\psi\(E,\nabla^{E}\), 
	s^{*}\delta_{X}$ on $X$ are	well defined such that 
	\begin{align}\label{eqWFsd}
	&{\rm WF}\(s^{*}\psi\(E,\nabla^{E}\)\)\subset N_{X'/X}^{*},&{\rm 
	WF}(s^{*}\delta_{X})\subset N^{*}_{X'/X}.
	\end{align}
	The following identity of currents on $X$ holds,
\begin{align}\label{eqsphi}
	d\(s^{*}\psi(E,\nabla^{E})\)=e(E,\nabla^{E})-s^{*}\delta_{X}. 
\end{align} 	
\end{prop}

\begin{re}\label{reBZ38}
 If $U\in C^{\infty}(X,TX)$ is a vector field on $X$ which has only isolated non degenerated 
zeros, i.e., 
 in a neighbourhood of a zero $x$ of $U$ there is a 
 system of coordinates  $y=(y_{1},\ldots,y_{m})$ and a matrix 
 $A$ with $\det 
 A\neq0$ such that  $x$ is represented by $y=0$ and 
 \begin{align}\label{eqU=Ay}
 	U(y)=Ay+\cO(|y|^{2}). 
 \end{align} 
By Proposition \ref{PropBZ38}, $U^{*}\psi(TX,\nabla^{TX})$ is a well 
defined current. Moreover, if $\e_{x}={\rm sgn}\det(A)$ is the Poincar\'e-Hopf 
index\footnote{If $U$ is Morse-Smale, then $(-1)^{m}\e_{x}=(-1)^{\ind(x)}.$ See the discussion 
after \eqref{eqindxEu} about the sign $(-1)^{m}$.}  at $x$, then 
$U^{*}\delta_{X}$ is a Radon measure  on $X$ defined  for $\mu\in 
C^{\infty}(X)$ by 
\begin{align}\label{eqUindd}
\int_{X}\mu 	\ U^{*}\delta_{X}=\sum_{x:\text{ zero of }U}\e_{x}\,\mu(x).  
\end{align} 
\end{re}

\subsection{Ray-Singer metric}\label{sRS}
We use the notation in Section \ref{S1M}. Recall that $(F,\nabla^{F})$ is a flat  vector bundle on $X$. 
Let $(\Omega^{{\scriptscriptstyle\bullet}}(X,F),d^{X})$ be the de Rham complex of smooth 
forms on $X$ with values in 
$F$.  By de Rham's theory, its  cohomology is $H^{{\scriptscriptstyle\bullet}}(X,F)$. 

Take metrics $g^{TX}$ and $g^{F}$ on $TX$ and $F$. Let 
$\<,\>_{\Lambda^{{\scriptscriptstyle\bullet}}(T^{*}X)\otimes F}$  be the induced metric on 
$\Lambda^{{\scriptscriptstyle\bullet}}(T^{*}X)\otimes F$.  Let $dv_{X}\in 
\Omega^{m}(X,o(TX))$  be the Riemannian volume form on $X$.  
For $s_{1},s_{2}\in \Omega^{{\scriptscriptstyle\bullet}}(X,F)$, set
\begin{align}\label{eqX}
	\<s_{1},s_{2}\>_{\Omega^{{\scriptscriptstyle\bullet}}(X,F)}=\int_{x\in X}\<s_{1}(x),s_{2}(x)\>_{\Lambda^{{\scriptscriptstyle\bullet}}(T^{*}X)\otimes 
	F}dv_{X}.
\end{align}
Then \eqref{eqX} defines an $L^{2}$-metric on $\Omega^{{\scriptscriptstyle\bullet}}(X,F)$. 

Let $d^{X*}$  be the formal adjoint of $d^{X}$ with respect to the 
$L^{2}$-metric $\<\cdot,\cdot\>_{\Omega^{{\scriptscriptstyle\bullet}}(X,F)}$. Put
\begin{align}
	\Box^{X}=d^{X}d^{X*}+d^{X*}d^{X}. 
\end{align}
Then $\Box^{X}$ is a formally self-adjoint  second order elliptic differential 
operator acting on $\Omega^{{\scriptscriptstyle\bullet}}(X,F)$. By Hodge theory, we have
\begin{align}\label{Hodge}
	\ker \Box^{X}\simeq H^{{\scriptscriptstyle\bullet}}(X,F). 
\end{align}
By \eqref{eqlambda} and \eqref{Hodge}, the restriction of the $L^{2}$-metric $\<\cdot,\cdot\>_{\Omega^{{\scriptscriptstyle\bullet}}(X,F)}$ to $\ker 
\Box^{X}$ induces a metric $|\cdot|^{RS,2}_{\lambda}$ on $\lambda$. 

Let $(\ker \Box^X)^{\bot}$ be the orthogonal  space to $\ker \Box^X$ 
in $\Omega^{{\scriptscriptstyle\bullet}}(X,F)$. Then $\Box^{X}$ acts as an invertible 
operator on $\(\ker\Box^{X}\)^{\bot}$. Let $\(\Box^{X}\)^{-1}$ be the 
inverse of $\Box^{X}$ acting on $\(\ker\Box^X\)^{\bot}$. Let 
$N^{\Lambda^{{\scriptscriptstyle\bullet}}(T^{*}X)}$ be the  number operator on 
$\Lambda^{{\scriptscriptstyle\bullet}}(T^{*}X)$, which is multiplication by $p$ on 
$\Lambda^{p}(T^{*}X)$. For $s\in \bC$ such that $\Re(s)>\frac{m}{2}$, set 
\begin{align}
	\zeta(s)=-\Tr\[(-1)^{N^{\Lambda^{{\scriptscriptstyle\bullet}}(T^{*}X)}}N^{\Lambda^{{\scriptscriptstyle\bullet}}(T^{*}X)}\(\Box^{X}\)^{-s}\].
\end{align}
By a result of Seeley \cite{Seeley66} or by \cite[Theorem 
7.10]{BZ92}, $\zeta(s)$ extends to a meromorphic function 
of $s\in \bC$, which is holomorphic at $s=0$.

\begin{defin}
The Ray-Singer metric $\|\cdot\|^{RS,2}_{\lambda}$ on 
$\lambda$ is defined  by 
\begin{align}
	\|\cdot\|^{RS,2}_{\lambda}=|\cdot|^{RS,2}_{\lambda}\exp\big(\zeta'(0)\big). 
\end{align}
\end{defin}

Let $\nabla^{TX}$ be the Levi-Civita connection on $(TX,g^{TX})$. Let
 $\psi(TX,\nabla^{TX})$ be the Mathai-Quillen current. 
By \eqref{eqMS} and by Proposition \ref{PropBZ38}, for any Morse-Smale vector field 
$V$, the pull back $V^{*}\psi(TX,\nabla^{TX})$ is a well-defined current of degree $m-1$ on $X$ with values in $o(TX)$. 
Set
\begin{align}
	\theta\(F,g^{F}\)=\Tr\[(g^{F})^{-1}\nabla^{F}g^{F}\]\in 
	\Omega^{1}(X).
\end{align}
Then, $\theta(F,g^{F})$ is a closed $1$-form and its cohomology class 
$\theta(F)=[\theta(F,g^{F})]\in H^{1}(X)$ does not depend on the 
metric $g^F$. Up to a normalisation, the class $\theta(F)$ coincides with  the  first  Kamber-Tondeur class \cite{KamberTondeur74}. 

The main result of Bismut-Zhang \cite[Theorem 0.2]{BZ92} is the following.

\begin{thm}\label{BZ}
	Suppose that  $f$ is a Morse function, whose gradient $\nabla f$ with respect to 
some Riemannian metric is Morse-Smale. The following identity holds, 
	\begin{align}\label{eqBZ}
	\log
	\(\frac{\|\cdot\|^{RS,2}_{\lambda}}{\|\cdot\|^{M,2}_{\lambda,-\nabla f}}\)=-\int_{X}\theta\(F,g^{F}\)\(\nabla f\)^{*}\psi\(TX, \nabla ^{TX}\).
\end{align}
\end{thm}

\subsection{A variation formula for certain characteristic 
form}\label{sMa}Let us follow \cite[Section VI.a)  and VI.b)]{BZ92}.  
Let  $(U_{\ell})_{0\l \ell\l 1}$ be a smooth family of vector fields 
on $X$, such that each $U_{\ell}$ has only isolated non degenerated 
zeros. By Proposition \ref{PropBZ38} and Remark \ref{reBZ38}, the integral 
\begin{align}
 \int_{X}\theta\(F,g^{F}\)U_{\ell}^{*}\psi\(TX, \nabla ^{TX}\)
\end{align} 
is well defined. Let us study its variation  with respect to $\ell\in [0,1]$. 

Let $q: [0,1]\times X\to X$ be the obvious projection. Consider a 
smooth section $U\in C^{\infty}([0,1]\times X, q^{*}(TX))$ defined by 
\begin{align}
	U: (\ell,x)\in [0,1]\times X\to U_{\ell}(x)\in T_{x}X. 
\end{align}
By the consideration after \eqref{eqX0}, the zero set of $U$ is a 
manifold of dimension $1$. 
Therefore, if $x_{1,0}, \ldots, x_{N,0}$ are the zeros of $U_{0}$, 
we can parametrize the zeros of $U_{\ell}$  by 
$x_{1,\ell},\ldots,x_{N,\ell}$ such that all the maps $\ell\in 
[0,1]\to x_{i,\ell}\in X$ are
smooth. Also, the Poincar\'e-Hopf index of $U_{\ell}$ at 
$x_{i,\ell}$ does not depend on $\ell$ and will be denoted by 
$\e_{i}\in \{\pm1\}$. The following proposition is a generalisation of \cite[Proposition 
6.4]{BZ92}. Since we will use this proposition several times in 
Section \ref{s33}, let us give a detailed proof. 


\begin{prop}\label{propBZVVa}
	The following identity holds, 
	\begin{align}\label{eqDDMA}
	\int_{X}\theta(F,g^{F})	 
		\Big(U_{1}^{*}\psi\(TX,\nabla^{TX}\)-	
		U_{0}^{*}\psi\(TX,\nabla^{TX}\)\Big)
		=\sum_{i=1}^{N}\e_{i}\int_{0}^{1}\theta\(F,g^{F}\)(\dot{x}_{i,\ell})d\ell.
	\end{align} 
\end{prop}
\begin{proof}
	Let us follow the proof of \cite[Proposition 6.1, Proposition 
	6.4]{BZ92}. Equip the pull back vector bundle 
$q^{*}(TX)$ over $[0,1]\times X$  with the pull back metric and the  pullback metric 
connection $\nabla^{q^{*}(TX)}$. Let 
$\psi\(q^{*}(TX),\nabla^{q^{*}(TX)}\)$ be the corresponding 
Mathai-Quillen current. By Proposition \ref{PropBZ38},  
$U^{*}\psi\(q^{*}(TX),\nabla^{q^{*}(TX)}\)$ and 
$U^{*}\delta_{[0,1]\times X}$ are  well defined currents such that 
\begin{align}\label{eqdphi01}
	d^{[0,1]\times X} \(	
	U^{*}\psi\(q^{*}(TX),\nabla^{q^{*}(TX)}\)\)=e\(q^{*}(TX),\nabla^{q^{*}(TX)}\)-U^{*}\delta_{[0,1]\times X}. 
\end{align} 
By our construction, 
\begin{align}\label{eqdphi02}
	e\(q^{*}(TX),\nabla^{q^{*}(TX)}\)=q^{*}e\(TX,\nabla^{TX}\).
\end{align} 
Since $\theta(F,g^{F})$ is a closed $1$-form on $X$, by 
\eqref{eqdphi01} and \eqref{eqdphi02}, we have 
\begin{align}
		d^{[0,1]\times X} \(q^{*}\theta(F,g^{F})	\wedge 
		U^{*}\psi\(q^{*}(TX),\nabla^{q^{*}(TX)}\)\)
		=q^{*}\theta\(F,g^{F}\) \wedge U^{*}\delta_{[0,1]\times X}. 
\end{align} 
Integrating the above formula over $[0,1]\times X$, by the Stokes 
formula,  we get \eqref{eqDDMA}. 
\end{proof}

\subsection{Proof of the main result}\label{s33}

 We restate our main 
result Theorem \ref{thm1}, which is an extension of Theorem \ref{BZ}.  
\begin{thm}\label{Thm1}
	Suppose that $V$ is a Morse-Smale vector field. The following identity holds,
	\begin{align}
	\log
	\(\frac{\|\cdot\|^{RS,2}_{\lambda}}{\|\cdot\|^{M,2}_{\lambda,V}}\)=-\int_{X}\theta\(F,g^{F}\)(-V)^{*}\psi\(TX,\nabla^{TX}\).
\end{align}
\end{thm}
\begin{proof}
Take 	$\(V_{\ell}\)_{0\l \ell\l 1}$ as in Proposition
\ref{propV01}. 
By Remark \ref{reDMA}, we have
\begin{align}\label{MVMV1}
	\|\cdot\|^{M,2}_{\lambda,V}=\|\cdot\|^{M,2}_{\lambda,V_{1}}.
\end{align}
Since the critical elements of $V$ and $V_{1}$ coincide, the fixed points of $V_{\ell}$ form  smooth loops on $X$. By Remark \ref{reVl}, we can assume that the fixed points of $V_{\ell}$ are in a small neighbourhood of 
the fixed points set of $V$.  In particular, the above loops are 
contractible. By Proposition \ref{propBZVVa} and by the closedness of $\theta(F,g^{F})$, we have
\begin{align}\label{XVV1}
	\int_{X}\theta\(F,g^{F}\)(-V)^{*}\psi\(TX,
	\nabla^{TX}\)=\int_{X}\theta\(F,g^{F}\)(-V_{1})^{*}\psi\(TX, \nabla^{TX}\).
\end{align}
By \eqref{MVMV1} and \eqref{XVV1}, it is enough to show our theorem
for the Morse-Smale vector field $V$ whose  critical elements are all 
of standard forms.
	
%

Take $f$  as in  Proposition
	\ref{propfV}.
%
By Proposition \ref{propMM} and Theorem  \ref{BZ}, we have
\begin{multline}\label{LLa}
	\log
	\(\frac{\|\cdot\|^{RS,2}_{\lambda}}{\|\cdot\|^{M,2}_{\lambda,V}}\)=-\int_{X}\theta\(F,g^{F}\)\(\nabla f\)^{*}\psi(TX, \nabla ^{TX})\\
	+\sum_{\gamma\in
		B}(-1)^{\ind(\gamma)}\log
		\left\|\det\tau(a_{\gamma}')\right\|^{2}_{\det
		F_{x'_{\gamma}}\otimes \(\det
		F_{x_{\gamma}}\)^{-1}}.
\end{multline}
By \eqref{LLa}, it remains to show
\begin{multline}\label{eqaaa}
	\int_{X}\theta\(F,g^{F}\)\(\nabla
	f\)^{*}\psi\(TX,
	\nabla^{TX}\)-\int_{X}\theta\(F,g^{F}\)(-V)^{*}\psi\(TX,\nabla^{TX}\)\\
=\sum_{\gamma\in
		B}(-1)^{\ind(\gamma)}\log
		\left\|\det\tau(a_{\gamma}')\right\|^{2}_{\det
		F_{x'_{\gamma}}\otimes \(\det
		F_{x_{\gamma}}\)^{-1}}.
\end{multline}
We assume that all the closed orbits are in standard form
of Case (1).  Cases (2)-(4) can be dealt similarly.

Following \cite[Section IV.c)]{BZ92},
choose a smooth triangulation $K$ of $X$ such that $A\cap
K^{m-1}=\varnothing$, and such that on $\ol{U}_{\gamma}\simeq \mathbb{S}^{1}\times
\ol{\mathbb{D}}^{m-1}$ the
triangulation is given by the $m$-simplex
$\sigma^{m}_{\gamma}=\(\mathbb{S}^{1}-\{1/4\}\)\times
\ol{\mathbb{D}}^{m-1}$ and $(m-1)$-simplex
$\sigma^{m-1}_{\gamma}=\{1/4\}\times \ol{\mathbb{D}}^{m-1}$.

On each simplex $\sigma\in K^{m}\backslash K^{m-1}$ of maximal degree, choose a
primitive $W_{0,\sigma}\in C^{\infty}(\sigma)$ of $\theta(F,g^{F})$, 
such that on $\sigma$ we have 
\begin{align}
	dW_{0,\sigma}=\theta\(F,g^{F}\).
\end{align}
Let $W_{0}$ be the locally integrable current  on $X$, such that
for each $\sigma\in  K^{m}\backslash K^{m-1}$ we have 
\begin{align}
W_{0}|_{\sigma}=W_{0,\sigma}.	
\end{align}
 By our construction of $K$, for $\gamma\in B$, the two points 
$x_{\gamma},x_{\gamma}'$, and the integral curve $a_{\gamma}'$ are in the same simplex $\sigma^{m}_{\gamma}\in
K^{m}$, so that
\begin{align}\label{W0W0}
	W_{0}(x_{\gamma}')-W_{0}(x_{\gamma})=\log
		\left\|\det\tau(a_{\gamma}')\right\|^{2}_{\det
		F_{x'_{\gamma}}\otimes \(\det
		F_{x_{\gamma}}\)^{-1}}.
\end{align}

Set
\begin{align}\label{eqW1}
W_{1}=\theta\(F,g^{F}\)-dW_{0}.
\end{align}
Then $W_{1}$ is  a closed
current  of degree $1$ on $X$ such that $\mathrm{Supp}(W_{1})\subset 
K^{m-1}$.  By \eqref{eqWFsd} and by $A\cap K^{m-1}=\varnothing$, $(-V)^{*}\psi(TX,\nabla^{TX})$ is 
smooth in the neighbourhood of the support of $W_{1}$, so that 
\begin{align}
	W_{1}\wedge (-V)^{*}\psi\(TX,\nabla^{TX}\).
\end{align} 
is a well defined current on $X$.  By \eqref{eqsphi}, \eqref{eqUindd}, and \eqref{eqW1}, we have
\begin{multline}\label{eqVtoW1}
	-\int_{X}\theta(F,g^{F}) (-V)^{*}\psi\(TX,\nabla^{TX}\)=\int_{X}W_{0}\ e\(TX,\nabla^{TX}\)
	-\sum_{x\in A}(-1)^{\ind(x)}W_{0}(x)\\
	-\int_{X}W_{1}\wedge (-V)^{*}\psi\(TX,\nabla^{TX}\).
\end{multline} 
Similar when $-V$ is replaced by $\nabla f$, we get
 \begin{multline}\label{eqnftoW1}
	-\int_{X}\theta(F,g^{F}) (\nabla f)^{*}\psi\(TX,\nabla^{TX}\)=\int_{X}W_{0}\ e\(TX,\nabla^{TX}\)
	-\sum_{x\in A}(-1)^{\ind(x)}W_{0}(x)\\
+\sum_{\gamma\in 
B}(-1)^{\ind(\gamma)}\(W_{0}(x_{\gamma})-W_{0}(x_{\gamma}')\)	-\int_{X}W_{1}\wedge (-V)^{*}\psi\(TX,\nabla^{TX}\).
\end{multline} 

By \eqref{W0W0}, \eqref{eqVtoW1}, and \eqref{eqnftoW1}, we see that \eqref{eqaaa} is equivalent to
\begin{align}\label{eqLast}
	\int_{X}W_{1} \wedge (\nabla
	f)^{*}\psi\(TX,\nabla^{TX}\)=\int_{X}W_{1}\wedge
	(-V)^{*}\psi\(TX,\nabla^{TX}\).
\end{align}
By \eqref{eqdf=V}, on any simplex in
$K^{m-1}$ other than $\sigma_{\gamma}^{m-1}$, we have $\nabla f=-V$.
By Remark \ref{reK}, near $\sigma_{\gamma}^{m-1}$, $\nabla
	f$ and $-V$ can be connected by a family of vector fields without	
	zero. Using the fact that $\mathrm{Supp}(W_{1})\subset K^{m-1}$, 
	and by a version of Proposition \ref{propBZVVa} where 
	$\theta(F,g^{F})$ is replaced by the closed current $W_{1}$, we get \eqref{eqLast}.
%
The proof of our  theorem is completed.
\end{proof}

\def\cprime{$'$}
\providecommand{\bysame}{\leavevmode\hbox to3em{\hrulefill}\thinspace}
\providecommand{\MR}{\relax\ifhmode\unskip\space\fi MR }
\providecommand{\MRhref}[2]{%
  \href{http://www.ams.org/mathscinet-getitem?mr=#1}{#2}
}
\providecommand{\href}[2]{#2}

\end{document}